\documentclass[12pt,twoside,reqno]{amsart}
\usepackage[colorlinks=true,citecolor=blue]{hyperref}
\usepackage{hyperref,mathptmx,color,amsmath, amssymb, amsfonts, amsthm, mathptmx, enumerate,mathrsfs, upgreek, orcidlink}
\usepackage{graphicx}
\usepackage{epstopdf}
\newtheorem{theorem}{Theorem}[section]
\newtheorem{lemma}[theorem]{Lemma}

\newtheorem{corollary}[theorem]{Corollary}

\theoremstyle{definition}
\newtheorem{definition}[theorem]{Definition}

\numberwithin{equation}{section}

\begin{document}
	\setcounter{page}{1}		
	\vspace*{2.0cm}
	\title[Embeddings of anisotropic Sobolev spaces]
	{Embeddings of anisotropic Sobolev spaces into spaces of anisotropic Hölder-continuous functions}
	\author[N. Chems Eddine,  
		D.D. Repov\v{s} 
	]{\bf
		Nabil Chems Eddine$^{1}$, 	
			Du\v{s}an D. Repov\v{s}$^{2,*}$
	}
	\maketitle
	\vspace*{-0.6cm}	
	\begin{center}
		{\footnotesize
		$^1$Laboratory of Mathematical Analysis and Applications, 
    	Department of Mathematics, 
		Faculty of Sciences, 
		Mohammed V University, 
		P.O. Box 1014, Rabat, Morocco. 
		E-mail: {\it nabil.chemseddine@um5r.ac.ma}\newline			
		$^2$Faculty of Education and Faculty of Mathematics and Physics, 
		University of Ljubljana 
		\& Institute of Mathematics, Physics and Mechanics, 
		1000 Ljubljana, Slovenia. 
		Email:  {\it dusan.repovs@guest.arnes.si}			
		}
	\end{center}		
	\vskip 4mm {\footnotesize \noindent {\it  Abstract.}		
 	We introduce a novel framework for embedding anisotropic variable exponent Sobolev spaces into spaces of anisotropic variable exponent Hölder-continuous functions within rectangular domains. We establish a foundational approach to extend the concept of Hölder continuity to anisotropic settings with variable exponents, providing deeper insight into the regularity of functions across different directions. Our results not only broaden the understanding of anisotropic function spaces but also open new avenues for applications in mathematical and applied sciences. \\
		
		
		}	
	
	\renewcommand{\thefootnote}{}
	\footnotetext{$^*$Corresponding author 	\par}
	
	% \date{\today}
	%  \tableofcontents
	
	\section{Introduction} \label{s1}
	
	The broadening of classical Sobolev embedding theories, encompassing both isotropic and anisotropic frameworks, as well as constant or variable exponents, has garnered considerable interest among researchers. These theories play a pivotal role in the analysis of partial differential equations linked to Sobolev spaces. This significance has prompted extensive investigation, as evidenced by studies referenced in the literature 
	\cite{Acerbi0}-- \cite{Troisi2}.	For  $\Omega \subset \mathbb{R}^N$ ($N\geq2$) rectangular domain, the anisotropic Sobolev space  $W^{1,\overrightarrow{p}(x)}(\Omega)$  consists of functions
	$u\in  L^{p_M(x)}(\Omega)$ with first-order distributional derivatives in  $L^{p_i(x)}(\Omega)$, i.e.,
	\begin{align*}
		\displaystyle
		W^{1,\overrightarrow{p}(x)}(\Omega)&=\{u\in L^{p_M(x)}(\Omega):
		\partial_{x_i}u\in L^{p_i(x)}(\Omega)\text{ for every } i=1,\dots,N\},
	\end{align*}	
	where  $\overrightarrow{p}:\overline\Omega\to\mathbb{R}^N$ is a vector function defined as
	$
	\overrightarrow{p}(x)=\left(p_1(x), \dots,p_N(x)\right)
	$
	with  $p_i \in C_+(\overline\Omega)$, satisfying $\displaystyle 1< p_i^-:= \text{ess inf}_{x\in \Omega}p_i(x)\leq p_i(x)\leq p_i^+:= \text{ess sup}_{x\in \Omega}p_i(x)<\infty$
	for each $i\in\{1,\dots,N\}$, and 
	$
	\displaystyle p_m(x)=\min_{1\leq i \leq N}\{p_i(x)\},\   \displaystyle p_M(x)=\max_{1\leq i \leq N}\{p_i\}$,
	$ \displaystyle p^+_m= \text{ess sup}_{x\in \Omega}p_m(x),$ and $\displaystyle p^+_M= \text{ess sup}_{x\in \Omega}p_M(x).$\par 
	This space is equipped with the following norm:
	$$
	\|u\|_{W^{1,\overrightarrow{p}(x)}(\Omega)}
	:=\|u\|_{L^{p_M(x)}(\Omega)}
	+\sum_{i=1}^N\left\|\partial_{x_i}u \right\|_{L^{p_i(x)}(\Omega)},
	$$
	
	For	$\overrightarrow{p}(\cdot)\in ( C_{+}^{\log}(\overline{\Omega}))^{N} $ such that $\displaystyle 1<p_m(x):=\min_{1\leq i \leq N}\{p_i(x)\}\leq p_M(x):=\max_{1\leq i \leq N}\{p_i(x)\}< \overline{P}^{\ast}(x) $ and $\overline{P}(x)<N,$ for every $x\in \overline{\Omega}$, where $\overline{P}^{\ast}({\color{red}\cdot})$ is the critical Sobolev exponent given by
	$$
	\displaystyle\overline{P}^{\ast}(x):=\begin{cases}
		\frac{N\overline{P}(x)}{N-\overline{P}(x)}&\text{if } \overline{P}(x)<N, \\
		+\infty &\text{if } \overline{P}(x)\geq N,
	\end{cases}
	$$
	where $\displaystyle\overline{P}(x):=\frac{N}{\sum_{i=1}^{N} \frac{1}{p_i(x)}}$ represents the harmonic mean of $\{p_i(x)\}$ and  $C_{+}^{\text{log}}(\overline{\Omega})$ denotes the set of functions $p\in C_{+}(\overline{\Omega})$ that satisfy the log-Hölder continuity condition
	$$\sup \left\{ |p(x)-p(y)| \log \frac{1}{|x-y|} : x,y\in \overline{\Omega}, 0< |x-y|<\frac{1}{2} \right\} < \infty.$$
	
       We highlight the most significant contributions to the understanding of this space by Fan \cite{fananiso}, who has obtained several key results, including a density theorem and embedding theorems for subcritical exponents. Specifically, for every $h \in C(\Omega)$ with $1 \leq h(x) < \overline{P}^{\ast}(x),$ for every $x \in \overline{\Omega}$, the space $W^{1,\overrightarrow{p}(x)}(\Omega)$ is continuously and compactly embedded into $L^{h(x)}(\Omega)$ (see, e.g., Fan \cite[Theorem 2.5]{fananiso}). \par 
	For $p_i(x)=N,$  $1\leq i\leq N,$ and  $x\in \Omega$, the space $W^{1,N}(\Omega)$ is continuously and compactly embedded into $L^{ h(x) }(\Omega),$ for every $h\in C(\Omega)$ with $1\leq h(x)< \infty,$  $x\in \overline{\Omega}$, but in general, $W^{1,N}(\Omega)\not\subseteq L^{\infty }(\Omega)$ (see, e.g., Ha\v{s}kovec and Schmeiser \cite{HaskovecSchmeiser}).\par 
	
	Finally, for $\overline{P}(x)>N,$ $x\in \overline{\Omega}$, Fan \cite{fananiso} proved that there exists $\beta \in (0,1)$ such that $W^{1,\overrightarrow{p}(x)}(\Omega)$ is continuously embedded into the space of $\beta$-Hölder-continuous functions:
	
	$$W^{1,\overrightarrow{p}(x)}(\Omega)\hookrightarrow C^{0,\beta}(\overline{\Omega}),$$
	
	and consequently, $W^{1,\overrightarrow{p}(x)}(\Omega)\hookrightarrow C^{0}(\overline{\Omega}).$	

	Inspired by ideas used in Morrey \cite{Morrey1,Morrey2} and R{\'a}kosn{\'\i}k \cite{Rakosnik,Rakosnik2}, we aim to extend in this work the results mentioned above - by proving that the anisotropic Sobolev space $W^{1,\overrightarrow{p}(\cdot)}(\Omega)$ is embedded in the space of $\overrightarrow{\beta}(\cdot)$-Hölder-continuous functions (as defined in Definition \ref{def1}). 		The following is the main result of this paper:
	
	\begin{theorem} \label{theo}
		Let $\Omega$ be a rectangular domain in $\mathbb{R}^N$ and suppose 
		that
		$\overrightarrow{p}(\cdot)\in ( C_{+}^{}(\overline{\Omega}))^{N} $ 
		is
		such that $\displaystyle N <p_m(x):=\min_{1\leq i \leq N}\{p_i(x)\}\leq p_M(x):=\max_{1\leq i \leq N}\{p_i(x)\}< \infty,$ for every $x\in \overline{\Omega}$. Set $\overrightarrow{\beta}(\cdot)\in ( C_{+}^{}(\overline{\Omega}))^{N} $ with $\beta_i$ satisfying \eqref{eqlem2}. Then there exists a continuous embedding 
		\[
		W^{1,\overrightarrow{p}(x)}(\Omega)\hookrightarrow C^{0, \overrightarrow{\beta}(x)}(\overline{\Omega}).
		\] 
	\end{theorem}	
	
	As it is well known, the existence of embedding $W^{1,\overrightarrow{p}(\cdot)}(\Omega)\hookrightarrow C^{0,\beta}(\overline{\Omega})$ implies that the solutions in the anisotropic Sobolev space have uniform Hölder continuity $(\beta)$ across all directions in the domain $\Omega$. 
	This  is well-suited for scenarios where the physical phenomenon exhibits uniform smoothness or regularity in all spatial directions (isotropic phenomena). Isotropic materials or processes, where properties are the same in all directions, are well-captured by this embedding. On the other hand, the embedding $W^{1,\overrightarrow{p}(\cdot)}(\Omega)\hookrightarrow C^{0, \overrightarrow{\beta}(\cdot)}(\overline{\Omega})$ allows for different Hölder exponents in different directions, capturing the anisotropic nature of regularity. Each direction can have its own decay rate, providing a more flexible and nuanced description of the function behavior. In many real-world scenarios, physical properties vary anisotropically. For instance, in composite materials or turbulent fluid flows, smoothness can differ significantly along different spatial axes. Therefore, anisotropic Hölder spaces are crucial for accurately representing and modeling features in applied sciences that exhibit varying degrees of regularity in different spatial directions. Now, we give some motivation for this in applied sciences:
	
	\begin{enumerate}
		\item \textbf{Physics - Anisotropic Materials:}
		\begin{itemize}
			\item \textbf{Example:}
			Consider a composite material composed of reinforcing fibers embedded within a matrix. The mechanical properties of such materials, including stiffness and strength, vary depending on the orientation of the fibers relative to applied forces. Anisotropic Hölder spaces play a crucial role in accurately modeling and predicting the material's behavior by accounting for these directional variations. This is essential for engineering applications where precise understanding of material response under different loading conditions is necessary for design optimization and structural integrity. For more detail, we refer the reader to Kar's book \cite{Kar}.
			
		\end{itemize}
		
		\item \textbf{Physics - Fluid Dynamics:}
		\begin{itemize}
			\item \textbf{Example:}Turbulent flows in fluid dynamics display varying degrees of regularity along different spatial directions. Anisotropic Hölder spaces are instrumental in analyzing these flows by capturing the directional variations in velocity field decay rates. This characterization aids in understanding complex flow phenomena, crucial for applications in aerospace, environmental modeling, and industrial processes.
 For further elaboration, we direct the reader to the work of Davidson \cite{Davidson}.
		\end{itemize}
		
		\item \textbf{Image Restoration - Edge Preservation:}
		\begin{itemize}
			\item \textbf{Motivation:} In image processing, edges or boundaries frequently exhibit diverse smoothness characteristics along various orientations. Anisotropic Hölder spaces offer a suitable framework for image restoration algorithms, allowing them to effectively preserve these directional features during denoising or reconstruction processes. This capability is essential for maintaining edge sharpness and structural details, contributing to the enhancement of image quality and accurate feature representation. For further details, one can refer to Moon and Stirling \cite{MoonStirling}, for example.
		\end{itemize}
		
		\item \textbf{Image Restoration - Texture Analysis:}
		\begin{itemize}
			\item \textbf{Motivation:} Textures in images often exhibit diverse regularity along different orientations. Anisotropic Hölder spaces play a crucial role in advancing image processing techniques, facilitating the preservation and analysis of these directional textures. This enables a more accurate representation and enhanced preservation of intricate textural details in various applications, such as medical imaging, remote sensing, and computer vision. For further in-depth exploration of this topic, interested reader can refer to Gonzalez and Woods \cite{GonzalezWoods} and  Ren et al. \cite{RenFangYan}.
		\end{itemize}
		
		\item \textbf{Medical Imaging - Fiber Tractography:}
		\begin{itemize}
			\item \textbf{Example:} In diffusion tensor imaging, which studies the diffusion of water molecules in biological tissues, the orientation of fiber tracts can vary in regularity. Anisotropic Hölder spaces are invaluable for modeling and analyzing these directional variations, providing a framework to accurately characterize the complex structure and organization of biological tissues. This enables a better understanding and interpretation of diffusion properties in applications such as neuroimaging and medical diagnosis.
			For a more comprehensive understanding of this topic, those interested are encouraged to consult Johansen-Berg and Behrens
			\cite{Johansen} and Shung et al. \cite{Shung}.
		\end{itemize}
	\end{enumerate}
	
	In conclusion, we can say
	that
	 the choice between $C^{0,\beta}(\overline{\Omega})$ and $C^{0, \overrightarrow{\beta}(x)}(\overline{\Omega})$ embeddings depends on whether the regularity of the phenomenon is uniform or direction-dependent. Ani\-sotro\-pic Hölder spaces have found particular significance in scenarios where anisotropy is inherent, providing for a more detailed and accurate representation for a wide range of phenomena in applied sciences.
	
	The paper is structured as follows: In Section \ref{prelim},  we introduce the basic concepts and preliminaries necessary for proving our main results. Section \ref{proof} is dedicated to a detailed proof of the main result (Theorem \ref{theo}), highlighting our key contributions. In Section \ref{counterexamples}, we discuss the limits of Theorem \ref{theo} through specific counterexamples. Finally, Section \ref{applications} illustrates the practical applications of our findings within the realm of physical sciences, underlining the real-world relevance of our theoretical work.	

	\section{Preliminaries}\label{prelim}
	For all background materials  we refer to  comprehensive monographs \cite{PRR, radulescu1}.
	To prove the main result established in Theorem \ref{theo}, we shall need to recall the following lemma.
	\begin{lemma}[Fan {\cite[Theorem 2.5]{fananiso}}] \label{lem1}
		Let $\Omega$ be a rectangular domain and 
		$\overrightarrow{p}(\cdot)\in ( C_{+}^{}(\overline{\Omega}))^{N} $ such that $\displaystyle N <p_m(x):=\min_{1\leq i \leq N}\{p_i(x)\}\leq p_M(x):=\max_{1\leq i \leq N}\{p_i(x)\}< \infty$ for every $x\in \overline{\Omega}$.
		Then there exists a continuous embedding 
		\[
		W^{1,\overrightarrow{p}(x)}(\Omega)\hookrightarrow C(\overline{\Omega}).
		\] 	
	\end{lemma}
	We now introduce the definition of the anisotropic space of $\overrightarrow{\beta}(\cdot)$-Hölder-continuous functions.
	\begin{definition} \label{def1}
		Let $\Omega \subset \mathbb{R}^N$ be a bounded domain. Let $\overrightarrow{\beta}(x)=(\beta_1(x),\dots,\beta_N(x)),$ where $0\leq \beta_i(x)\leq 1,$ for $i=1,\dots,N$ and $\beta_i(x)\in C(\Omega)$. We define the anisotropic space $C^{0,\overrightarrow{\beta}(x)}(\overline{\Omega})$ of $\overrightarrow{\beta}(x)$-Hölder-continuous functions as the set of $u\in C(\overline{\Omega})$ with a finite norm given by
		\begin{equation} \label{eqN1}
			\|u\|_{C^{0,\overrightarrow{\beta}(x)}(\overline{\Omega})}
			:= \|u\|_{C(\overline{\Omega})}+ \sup_{x,y\in \Omega,\\ x\neq y } \dfrac{| u(x)-u(y)|}{\sum_{i=1}^{N}|x_i - y_i |^{\beta_i(x,y)}},
		\end{equation}
		where $\beta_i(x,y)= \min\{\beta_i(x),\beta_i(y)\}$.
	\end{definition}
	
	The space $C^{0,\overrightarrow{\beta}(x)}(\overline{\Omega})$ equipped with the norm \eqref{eqN1} forms a Banach space. 	
	
	\begin{lemma} \label{lem2}
		Let $\mathcal{Q}= (O,1)^N$ and
		$\overrightarrow{p}(\cdot)\in ( C_{+}^{}(\overline{\Omega}))^{N} $ such that $\displaystyle N <p_m(x):=\min_{1\leq i \leq N}\{p_i(x)\}\leq p_M(x):=\max_{1\leq i \leq N}\{p_i(x)\}< \infty,$ \  for every $x\in \overline{\Omega}$. Set
		
		\begin{equation} \label{eqlem2}
			\beta_i(x) =  \dfrac{	1 - \sum_{j=1}^{N}\frac{1}{p_j(x)}}{1 - \sum_{j=1}^{N}\frac{1}{p_j(x)}+ \frac{N}{p_1(x)}}
			, ~ i=1,\dots,N,  \text{ for every } x \in \mathcal{Q}.
		\end{equation}

		Then there exists a constant $c>0$ such that 
		
		$$
		| u(x)-u(y)| \leq c \|u\|_{W^{1,\overrightarrow{p}(x)}(\mathcal{Q})}  \sum_{i=1}^{N}|x_i - y_i |^{\beta_i(x,y)},
		$$
		where $\beta_i(x,y)=\min(\beta_i(x),\beta_i(y))$.
	\end{lemma}
	
	\begin{proof}
		Let  $u \in C^{\infty}(\overline{\mathcal{Q}})$ and $x, y \in \overline{\mathcal{Q}}$, such that $\displaystyle \sum_{i=1}^{N}|x_i - y_i |^{\beta_i(x,y)}=s$ to apply Lemma \ref{lem1}. We shall consider two cases.\\
		\textbf{Case 1:} $s \geq 1$\par 
		In this case, we directly apply Lemma \ref{lem1}, which ensures that the function $u$ belongs to an anisotropic variable exponent Sobolev space that is continuously embedded into the space of continuous functions. Therefore, the norm of $u$ can be bounded in terms of its anisotropic variable exponent Sobolev norm. Specifically, using the result from Lemma \ref{lem1}, we have:
		$$| u(x) - u(y) | \leq c_1 \| u \|_{W^{1,\overrightarrow{p}(x)}(\mathcal{Q})} \sum_{i=1}^{N} |x_i - y_i|^{\beta_i(x,y)}. $$
		This inequality holds for every $x, y \in \mathcal{Q}$, where $\mathcal{Q} = (0,1)^N$.
		
		\textbf{Case 2:}  $0<s<1$\\
		Here, we shall consider the scenario where $s<1 $.  We introduce a set $\Omega(s)\subset \mathcal{Q}$, which is a translation of a rectangular box $\displaystyle \prod_{i=1}^{N}\Big(0,s^{\frac{1}{\beta_i^{-}}}\Big)$ where $\beta_i^{-} = \text{min} \{\beta_i(x,y) \mid x, y \in \overline{\mathcal{Q}}\}$ . The choice of $\Omega(s)$ ensures that $x, y \in \Omega(s)$. \par
		
		Now, for $z \in \Omega(s)$, By the fundamental theorem of calculus and the triangle inequality, we have
		\[
		| u(x)- u(z)|=\left| \int_{0}^{1} \frac{d}{dt} u(x + t(z-x)) \, dt \right| \leq \int_{0}^{1} \left| \frac{d}{dt} u(x + t(z-x)) \right| \, dt.
		\]
		Using the chain rule, the derivative $\frac{d}{dt} u(x + t(z-x))$ can be expressed as
		
		\[
		\frac{d}{dt} u(x + t(z-x)) = \sum_{i=1}^{N} \partial_{x_i} u(x + t(z-x)) \cdot \partial_{t}(x + t(z-x))_i,
		\]		
		where
		 $\partial_{x_i} u(x + t(z-x))$ represents the partial derivative of $u$ with respect to the \(i\)th component evaluated at $x + t(z-x)$, and $\partial_{t}(x + t(z-x))_i$ is the \(i\)th component of the vector $(x + t(z-x))$. Note that $\partial_{t}(x + t(z-x))_i = z_i - x_i$. Substituting this expression back into the integral, we obtain
		\[
		| u(x)- u(z)| \leq \sum_{i=1}^{N} \int_{0}^{1} \left| \partial_{x_i} u(x + t(z-x)) \right| \cdot |z_i - x_i|\, dt.
		\]
		Since $z \in \Omega(s)$, we have $|z_i - x_i| \leq s^{\frac{1}{\beta_i^{-}}}$ for every $i$. Therefore, we obtain
		$$
		| u(x)- u(z)| \leq \sum_{i=1}^{N} s^{\frac{1}{\beta_i^{-}}} \int_{0}^{1} \Big|\partial_{x_i} u(x + t(z-x)) \Big|
		dt. $$
		Denote $\displaystyle \sum_{i=1}^{N}\frac{1}{\beta_i^{-}}=\sigma$ we have
		$\text{meas}(\Omega(s))= s^{\sigma}$,
		so that
		\begin{align}\label{inq1}
			\begin{aligned}
				| u(x) - s^{-\sigma}	\int_{\Omega(s)} u(z)dz|&\leq		 s^{-\sigma}	\int_{\Omega(s)} |u(x) -u(z)|dz\\
				&\leq \sum_{i=1}^{N} s^{ -\sigma + \frac{1}{\beta_i^{-}}}
				\int_{\Omega(s)}\int_{0}^{1}\Big|\partial_{x_i} u(x + t(z-x))\Big|
				dt dz.
			\end{aligned}
		\end{align}
		Setting $\Omega(s,t)= \{ x + t(z-x); z \in \Omega(s)\}$, we have $\text{meas}(\Omega(s,t))= t^N s^\sigma$. Applying Fubini's theorem successively, along with the substitution $z\to x+t(z-x)$ and Hölder's inequality in \eqref{inq1}, yields
		\begin{align*}
			| u(x) - s^{-\sigma}	\int_{\Omega(s)} u(z)dz|&\leq 
			\sum_{i=1}^{N} s^{ -\sigma+ \frac{1}{\beta_i^{-}}}
			\int_{0}^{1} t^{-N}
			\int_{\Omega(s,t)}|\partial_{x_i} u(z)|
			dzdt\\
			&\leq \sum_{i=1}^{N} s^{ -\sigma + \frac{1}{\beta_i^{-}}}
			\|\partial_{x_i} u\|_{L^{p_i(x)}(\mathcal{Q})}
			\int_{0}^{1} t^{-N} \Big(\text{meas}(\Omega(s,t))\Big)^{1-\frac{1}{p_i(x)}}dt\\
			&\leq \sum_{i=1}^{N} s^{\frac{1}{\beta_i^{-}}- \frac{\sigma}{p_i(x)} }
			\|\partial_{x_i} u\|_{L^{p_i(x)}(\mathcal{Q})}
			\int_{0}^{1} t^{-\frac{N}{p_i(x)}}dt\\
			&\leq  \sum_{i=1}^{N}s
			\|\partial_{x_i} u\|_{L^{p_i(x)}(\mathcal{Q})}
			\int_{0}^{1} t^{-\frac{N}{p_i(x)}}dt.
		\end{align*}
		If we use the fact $N< p_i(x)$ for every $x\in \Omega$  and apply the same procedure to $y$, we obtain
		\begin{align*}
			| u(x) -  u(y)|&\leq 2 c_2 s \sum_{i=1}^N 	\|\partial_{x_i} u\|_{L^{p_i(x)}(\mathcal{Q})} \leq
			2c_2 	\| u\|_{W^{1,\overrightarrow{p}(x)}(\mathcal{Q})} s.
		\end{align*} 
		This completes the proof of Lemma~\ref{lem2}.
	\end{proof}
	\section{Proof of Theorem~\ref{theo}}\label{proof}	
	
	\begin{proof}
		For any $u \in C^{\infty}(\overline{\Omega})$, by Lemma \ref{lem1}, there exists a constant $c_1$
		such  that 
		\begin{equation} \label{eqt1}
			\|u\|_{C(\overline{\Omega})} \leq c_1\|u\|_{W^{1,\overrightarrow{p}(x)}(\Omega)}.
		\end{equation}
		
		Since  $\Omega$ is a rectangular domain in $\mathbb{R}^N$, 
		it follows that for every $x\in \Omega,$ there is a closed cube $\mathcal{Q}$ with edge length $H$ whose edges are parallel to the coordinate axes and one vertex is at the origin, such that $y+\mathcal{Q}\subset \Omega$ for every $y\in \Omega$, $|y-x|<\delta$.
		
		Now, let  $u \in C^{\infty}(\overline{\Omega})$ and $x, y\in \Omega$. Set 
		\begin{equation}\label{key1}
		\displaystyle	K=\min\{H,\delta,1\}.
		\end{equation}
		
		If $\sum_{i=1}^N |x_i -y_i |^{\beta_i(x,y)} \geq K$, then \eqref{eqt1} yields
		\begin{align}\label{eqt2}
			\begin{aligned}
				| u(x) - u(y) |&\leq |u(x) | + | u(y)|\leq 2c_1 \|u\|_{W^{1,\overrightarrow{p}(x)}(\Omega)}\\
				&\leq 
				2c_1 K^{-1}\|u\|_{W^{1,\overrightarrow{p}(x)}(\Omega)}\sum_{i=1}^{N}|x_i -y_i |^{\beta_i(x,y)}.
			\end{aligned}
		\end{align}
		Suppose now that $\displaystyle \sum_{i=1}^{N}|x_i -y_i |^{\beta_i(x,y)} <K$. Then by \eqref{key1}, we have
		$$
		|x-y|\leq \sum_{i=1}^{N}|x_i-y_i|^{\beta_i(x,y)}< \delta
		$$
		and there exists a cube $\mathcal{Q}$ described above. In particular,  $x+\mathcal{Q}\subset \Omega$, $y+\mathcal{Q}\subset \Omega$. There exist numbers $t_1,\dots,t_n\in \{-1,1\}$ 
		such that
		$$ \displaystyle \mathcal{Q} =\{ z; 0\leq t_iz_i \leq H \text{ for } i=1,\dots,N\}.$$
		Set $\displaystyle z_i =t_i \max(t_ix_i, t_iy_i)$, $z=(z_1,\dots,z_n)$. Since $|x-y|<K\leq H$, we have $z\in (x+\mathcal{Q})\cap (y+\mathcal{Q})$ and 
		
		$$
		|x_i-z_i|+ |z_i-y_i|= |x_i-y_i|, ~~ i=1,\dots,N.
		$$
		Applying Lemma \ref{lem2} to the cubes $x+\mathcal{Q}$, $y+\mathcal{Q}$, we obtain
		
		\begin{align*}
			| u(x) - u(y) |&\leq 	| u(x) - u(z) |+ 	| u(z) - u(y) |\\
			&\leq  c_2\|u\|_{W^{1,\overrightarrow{p}(x)}(\Omega)} \Big(
			\sum_{i=1}^{N}|x_i -z_i |^{\beta_i(x,y)}
			+ \sum_{i=1}^{N}|z_i -y_i |^{\beta_i(x,y)}  \Big)\\
			&\leq  c_3\|u\|_{W^{1,\overrightarrow{p}(x)}(\Omega)}
			\sum_{i=1}^{N}|x_i -y_i |^{\beta_i(x,y)},
		\end{align*}
		where $\displaystyle c_3= c_2\max_{1\leq i \leq N}2^{1-\beta_i(x,y)}$. 
		Setting $C= \max\{2c_1K^{-1},c_3\}$  we obtain 
		$$\|u\|_{W^{1,\overrightarrow{\beta}(x)}(\overline{\Omega})}
		\leq C\|u\|_{W^{1,\overrightarrow{p}(x)}(\Omega)}.$$
		This completes the proof of Theorem \ref{theo}.
	\end{proof}
	
 Next, we present two corollaries of Theorem \ref{theo}. The first one reveals that with a uniform Sobolev exponent, the embedding simplifies to isotropic Hölder continuity. The second highlights that higher Sobolev exponents lead to increased Hölder continuity, indicating enhanced function regularity within the domain.  These insights bridge the gap in understanding between anisotropic Sobolev spaces and their isotropic Hölder counterparts, emphasizing the impact of Sobolev exponents on the regularity of embedded functions.	
	
	\begin{corollary}  \label{cor1}
	If $ \overrightarrow{p}(x) = (p(x), \dots, p(x))$ (uniform for all directions), then \( W^{1,\overrightarrow{p}(x)}(\Omega) \hookrightarrow C^{0,\beta(x)}(\overline{\Omega}) \), where $ \beta(x) =1 - \frac{N}{p(x)}$  is a scalar representing the isotropic Hölder continuity exponent.
	\end{corollary} 
	\begin{proof}
	Assume that $ \overrightarrow{p}(x) = (p(x), \dots, p(x))$ for every $ x \in \Omega $, the anisotropic nature vanishes, and the space $ W^{1,\overrightarrow{p}(x)}(\Omega)$ reduces to $ W^{1,p(x)}(\Omega)$. In this case, the anisotropic Hölder exponents  $\overrightarrow{\beta}(x)$ simplify as follows:

	$$\beta_i(x) = \frac{1 - N/p(x)}{1 - N/p(x) + N/p(x)} = 1 - N/p(x),$$
	for every $i = 1, \dots, N$,
	 since all $ p_i(x)$ are equal and the anisotropic terms cancel out. Therefore, the Hölder exponent $ \beta(x)$ is uniform in all directions, and the space $ C^{0,\overrightarrow{\beta}(x)}(\overline{\Omega})$ simplifies to the isotropic Hölder space $ C^{0,\beta(x)}(\overline{\Omega})$.
	
	Thus, the embedding theorem for anisotropic spaces simplifies to the classical embedding theorem for isotropic spaces, where $ W^{1,p(x)}(\Omega)$ is embedded into $ C^{0,\beta(x)}(\overline{\Omega})$ with  $\beta(x) = 1 - N/p(x)$, as long as $p(x) > N$.	\end{proof}
\begin{corollary}
	If $ p_i(x) > N $ for every $i$ in  $ \overrightarrow{p}(x) = (p_1(x), p_2(x), \dots, p_N(x))$, then the Hölder continuity exponent $ \overrightarrow{\beta}(x)$ increases, indicating higher regularity in $C^{0,\overrightarrow{\beta}(x)}(\overline{\Omega})$.
\end{corollary}
\begin{proof}
	
	Given $ \overrightarrow{p}(x) = (p_1(x), \dots, p_N(x))$ with $ p_i(x) > N $ for every  $ i \in \{1, \dots, N\} $ and  $ x \in \Omega $, we consider the embedding $ W^{1,\overrightarrow{p}(x)}(\Omega) \hookrightarrow C^{0,\overrightarrow{\beta}(x)}(\overline{\Omega})$ as established earlier.
	
	The Hölder exponents $ \overrightarrow{\beta}(x)$ in each direction are determined by the formula provided in the following theorem:
	
	$$
	\beta_i(x) = \frac{1 - \sum_{j=1}^{N}\frac{1}{p_j(x)}}{1 - \sum_{j=1}^{N}\frac{1}{p_j(x)} + \frac{N}{p_i(x)}}, \quad i = 1, \dots, N.
	$$
	
	Since $ p_i(x) > N $ for every $ i $, we have $ \frac{1}{p_i(x)} < \frac{1}{N} $. Therefore,
	
	$$
	\sum_{j=1}^{N}\frac{1}{p_j(x)} < \sum_{j=1}^{N}\frac{1}{N} = 1,
	$$
	
	which implies that \( 1 - \sum_{j=1}^{N}\frac{1}{p_j(x)} > 0 \).
	
	As $p_i(x)$ increases significantly beyond $ N $, $ \frac{N}{p_i(x)} $ approaches 0, and hence $ \beta_i(x)$ approaches 1, which is the maximum value for a Hölder exponent indicating Lipschitz continuity.
	
	Therefore, for higher exponents $p_i(x)$, the Hölder exponent $ \beta_i(x)$ increases, leading to improved regularity in the embedding. Specifically, functions in $ W^{1,\overrightarrow{p}(x)}(\Omega)$ exhibit higher regularity in the sense of $\overrightarrow{\beta}(\cdot)$-Hölder continuity, demonstrating enhanced smoothness in such anisotropic Sobolev spaces.
\end{proof}	
	
\section{Counterexamples for Domains with General Shapes} \label{counterexamples}

This section delves into the intricacies of extending the embeddings from anisotropic Sobolev spaces $W^{1,\overrightarrow{p}(x)}(\Omega)$ to spaces of anisotropic Hölder-continuous functions $C^{0, \overrightarrow{\beta}(x)}(\overline{\Omega})$ within domains that do not adhere to the rectangular constraint as stipulated in Theorem \ref{theo}. We explore the domain $\Omega$ within $\mathbb{R}^2$ and scrutinize a function $u(x,y)$ to shed light on the profound challenges that emerge when dealing with non-rectangular, irregularly shaped domains, particularly those exhibiting cusps or pronounced features.

\subsection{Counterexample 1: Domain with a Mild Cusp}

Consider a domain $\Omega$ within $\mathbb{R}^2$ characterized by a cusp, defined as $\Omega = \{(x,y) \in \mathbb{R}^2 : x^2 < y < 2x^2, 0 < x < 1\}$. For the sake of illustration, we adopt a simplified scenario where $\overrightarrow{p}(x) = (p(x), p(x))$ with $p(x) = 4$ being a constant function. The unique and irregular geometry of $\Omega$, particularly the cusp at the origin, introduces significant complications for the Sobolev embedding. The conventional methodologies employed to extend functions beyond $\Omega$ or to approximate functions within $W^{1,p(x)}(\Omega)$ using smooth functions are rendered highly intricate. In this context, the Sobolev embedding theorem might suggest that functions in $W^{1,4}(\Omega)$ should naturally embed into some Hölder space. Nonetheless, the vicinity of the cusp witnesses erratic behaviors in functions and their derivatives due to the sharp geometric turn. For instance, contemplate a function $u$ within $\Omega$ that follows $u(x,y) = \sqrt{x}$ along a trajectory nearing the cusp. Despite $u$ being part of $L^4(\Omega)$ and its derivative $\partial_x u = \frac{1}{2\sqrt{x}}$ lying in $L^4(\Omega)$ due to the proximity to the cusp, $u$ fails to display uniform Hölder continuity across the domain because of the singularity at the origin.

\subsection{Counterexample 2: Domain with a Pronounced Cusp}

Enhancing the complexity, let's examine a domain $\Omega$ delineated by a more pronounced cusp than in the initial example: $\Omega = \{ (x,y) \in \mathbb{R}^2 : y > |x|^\alpha, -1 < x < 1, 0 < y < 1 \}$, with $\alpha > 1$ to intensify the cusp at the origin, thus making it more conspicuous. The exponent vector function $\overrightarrow{p}(x) = (p_1(x,y), p_2(x,y))$ is defined as $p_1(x,y) = 2 + \sin(\pi x)$ and $p_2(x,y) = 2 + \cos(\pi y)$, introducing a spatial variation in the Sobolev exponents and rendering the space both anisotropic and variable-exponent. This setting ensures that the exponents stay within a conventional range for Sobolev spaces, thus evading trivial scenarios.

The function $u: \Omega \rightarrow \mathbb{R}$ is defined as $u(x,y) = \sqrt{y - |x|^\alpha}$. This function remains smooth across $\Omega$ except at the cusp and is crafted to exhibit the irregular behavior near the cusp. The first-order partial derivatives are $\partial_{x} u = -\frac{\alpha x |x|^{\alpha-2}}{2\sqrt{y - |x|^\alpha}}$ and $\partial_{y}u = \frac{1}{2\sqrt{y - |x|^\alpha}}$. Considering $p_M(x) \leq 3$ for every $(x,y) \in \Omega$ and the bounded nature of $\Omega$, $u(x,y)$ remains bounded within $\Omega$, implying that $u(x,y)^{p_M(x)}$ is also bounded and integrable over $\Omega$, signifying $u \in L^{p_M(x)}(\Omega)$.

Assessing whether $\partial_{x}u\in L^{p_1(x,y)}(\Omega)$ and $\partial_{y}u \in L^{p_2(x,y)}(\Omega)$ presents a nuanced challenge. Near the cusp $(x,|x|^\alpha)$, $\partial_{x}u$ behaves akin to $x^{(\alpha-3)/2}$, which is integrable over $\Omega$ given
that  $\alpha > 1$ and the domain's restriction of $x$ to the finite interval $(-1, 1)$. Consequently, $\partial_{x}u$ resides within $L^{p_1(x,y)}(\Omega)$ since $p_1(x,y) \geq 2$. Conversely, $\partial_{y}u$, being less singular, is also integrable over the bounded domain $\Omega$, ensuring its inclusion in $L^{p_2(x,y)}(\Omega)$.

The crux of the matter lies in the anisotropic Hölder continuity of $u(x,y) = \sqrt{y - |x|^\alpha}$, especially near the cusp along $y = |x|^\alpha$. As $(x,y)$ gravitates toward the origin, the derivatives $\partial_{x}u$ and $\partial_{x}u$ become unbounded, underscoring a failure to adhere to the uniform rate of change requisite for Hölder continuity. This issue is particularly pronounced near the cusp, where the path-dependent approach of $u$ toward 0 complicates the determination of a consistent Hölder constant $K$. Therefore, $u$ may not satisfy the criteria for anisotropic Hölder continuity in $C^{0, \overrightarrow{\beta}(x)}(\overline{\Omega})$, especially in the vicinity of singularities such as cusps.

These counterexamples highlight the significant challenges and limitations in applying the embedding theorem (Theorem \ref{theo}) to domains with non-standard boundaries. They emphasize the need for careful consideration and further research in this area. Consequently, we conclude this section by posing an important question: How can the embedding results be generalized or adapted to include more complex, non-rectangular domains, especially those featuring irregular boundaries or internal cusps?	
	
 \section{Applications of Anisotropic Sobolev Embeddings in Physical Sciences}\label{applications}
	
This section delves into two distinct applications of the embeddings from anisotropic Sobolev spaces to anisotropic Hölder spaces, as elucidated by Theorem \ref{theo}. These applications are situated in the realms of heat conduction in composite materials and turbulent fluid flow in anisotropic porous media, showcasing the theorem's versatility in analyzing physical phenomena characterized by directional dependencies.

\subsection{Application 1: Heat Conduction in Composite Materials} 

Composite materials are distinguished by their anisotropic properties, resulting from the combination of two or more constituent materials with different physical characteristics. This anisotropy introduces directional dependencies in properties like thermal conductivity, complicating the analysis of heat conduction. To explore this, we consider a rectangular domain $\Omega = [0,1] \times [0,2] \subset \mathbb{R}^2$, symbolizing a cross-section of a composite material. The domain's anisotropic thermal conductivity is depicted by the exponent vector $\overrightarrow{p}(x) = (3 + \sin(\pi x), 2 + 0.5\cos(\pi y))$, reflecting the material's layered structure or fiber orientations with spatial variations.

The temperature distribution within the composite is modeled by $u(x,y) = \sin(\pi x) \cdot \exp(-y)$, where $\sin(\pi x)$ might indicate periodic variations in temperature due to heterogeneity along the $x$-direction, while $\exp(-y)$ captures a temperature gradient along the $y$-direction, perhaps due to cooling effects from a boundary. Leveraging Theorem \ref{theo} enables quantification of the temperature distribution's smoothness across different directions in the composite material, a critical aspect for designing materials with specific thermal properties and for predicting heat flow behavior in composite structures. This analysis is invaluable in fields such as aerospace and automotive engineering, where tailored thermal management is crucial.

\subsection{Application 2: Turbulent Fluid Flow in Anisotropic Porous Media}

Turbulent fluid flow through anisotropic porous media plays a critical role in various environmental and industrial scenarios, including groundwater movement through aquifers and oil extraction from geological reservoirs. The inherent anisotropy of these media, characterized by heterogeneous pore sizes, shapes, and orientations, exerts a profound influence on the dynamics of fluid flow. To model such a scenario, consider the domain $\Omega = [0,1] \times [0,3] \subset \mathbb{R}^2$, which represents a vertical cross-section of anisotropic porous media, encapsulating the complexities of vertical stratification and horizontal variability in permeability.

The exponent vector $\overrightarrow{p}(x) = (4 - 0.1x, 3 + 0.2y)$ is chosen to reflect the nuanced variations in flow permeability within the media. The term $4 - 0.1x$ signifies a gradual decrease in horizontal permeability, perhaps due to the accumulation of finer sediments or the presence of less permeable layers. Conversely, $3 + 0.2y$ captures the increase in vertical permeability, a common feature in stratified deposits where layers of varying permeability influence the vertical flow dynamics.

To depict the turbulent flow's horizontal velocity component within this medium, the function $u(x,y) = \sin(2\pi x) \cdot y^2 e^{-y}$ is employed. The oscillatory behavior introduced by $\sin(2\pi x)$ mimics the turbulence-induced velocity fluctuations, while the term $y^2 e^{-y}$ models the amplification and subsequent attenuation of the flow with depth, capturing the essence of flow dynamics in real-world porous structures.

The application of Theorem \ref{theo} in this context is invaluable, offering a rigorous framework to analyze the anisotropic regularity of the velocity distribution within the porous media. Such an analytical approach is indispensable for accurately predicting fluid flow patterns in anisotropic environments, a key factor in optimizing groundwater management strategies, enhancing oil recovery techniques, and designing efficient environmental remediation processes. This detailed understanding of fluid behavior in complex media not only aids in addressing practical challenges but also contributes to the theoretical advancement in the study of fluid dynamics in porous and anisotropic materials.

These applications underscore the significance of Theorem \ref{theo} in providing a rigorous mathematical framework to study and quantify anisotropic regularity in physical phenomena, offering valuable insights into their behavior and properties across different directions.	

\subsection*{Acknowledgements}
The authors thank the referees for comments and suggestions.
	
\subsection*{Author contributions}
All authors contributed to the study conception, design, material
preparation, data collection, and analysis. All authors read and
approved the final manuscript.

\subsection*{Conflict of interest}
The authors state that there is no conflict of interest.

\subsection*{Ethical approval}
The conducted research is not related to either human or animal use.

\subsection*{Data availability statement} Data sharing is not applicable to this article as no data sets were generated or
analyzed during the current study.
	
\subsection*{Funding}  	
The second author acknowledges the funding received from the Slovenian Research and Innovation Agency grants P1-0292, J1-4031, J1-4001, N1-0278, N1-0114, and N1-0083.

\end{document}